\documentclass[a4paper]{article}
\usepackage[margin=1in]{geometry}

\usepackage{graphicx} 
\usepackage{amsthm,amsmath}
\usepackage{tikz}
\usepackage{authblk} 
\usepackage{enumerate}
\usepackage{subcaption}
\usepackage{amssymb}
\usepackage{multirow}
\usepackage{adjustbox}
     \newcommand{\card}[1]{\left\vert #1 \right\vert}
    \newcommand{\set}[1]{\left\{ #1 \right\}}
\title{Maximal Independent Sets in Planar Triangulations}

\usepackage{xcolor}
\usepackage{hyperref}
	\hypersetup{
	    colorlinks=true,
	    linkcolor=red,
	    citecolor=red,
	    pdfborder={0 0 0}}
\author{P. Francis}

\affil{Department of Mathematics\\
    SAS, VIT-AP University, Amaravati, 
    Andhra Pradesh, India.\\
    francis@vitap.ac.in}

\author{Abraham M. Illickan}
    
\affil{Department of Computer Science\\
    University of California, Irvine\\
    aillicka@uci.edu}
   \author{Lijo M. Jose \thanks{Supported by Half-Time Research Assistantship (HTRA) funded by Ministry of Education, Government of India}}
   \author{Deepak Rajendraprasad}
\affil{  Department of Computer Science\\
    Indian Institute of Technology Palakkad\\
    112004005@smail.iitpkd.ac.in,
    deepak@iitpkd.ac.in}

\newtheorem{theorem}{Theorem}
\newtheorem{lemma}[theorem]{Lemma}
\newtheorem{observation}[theorem]{Observation}
\theoremstyle{remark}
\newtheorem{claim}{Claim}[theorem]

\pgfdeclarelayer{nodelayer}
\pgfdeclarelayer{edgelayer}
\pgfdeclarelayer{facelayer}
\pgfsetlayers{edgelayer,nodelayer,facelayer,main}

\usetikzlibrary{fit}

\tikzstyle{nofill_node}=[fill=white, draw=black, shape=circle,inner sep = 0,minimum size = 0.8cm,]
\tikzstyle{gray}=[-, draw={gray}]
\tikzset{
    mygauss/.pic={
        \begin{tikzpicture}
	\begin{pgfonlayer}{nodelayer}
		\node [style={nofill_node}] (0) at (0, 3) {};
		\node [style={nofill_node}] (1) at (-4, -2) {};
		\node [style={nofill_node}] (2) at (4, -2) {};
		\node [style={nofill_node}] (3) at (0, 1.5) {};
		\node [style={nofill_node}] (4) at (-1, 0) {\tiny$v_2$};
		\node [style={nofill_node}] (5) at (1, 0) {\tiny$v_3$};
		\node [style={nofill_node}] (6) at (0, -1) {\tiny$v_1$};
		\node  (7) at (1.25, 3.5) {\tiny{$y = b_2 = b_3$}};
		\node  (8) at (1, 2) {\tiny{$x=b_{23}$}};
	\end{pgfonlayer}
	\begin{pgfonlayer}{edgelayer}
		\draw (1) to (0);
		\draw (0) to (2);
		\draw (1) to (2);
		\draw (0) to (3);
		\draw [style=gray] (4) to (5);
		\draw [style=gray] (5) to (6);
		\draw [style=gray] (6) to (4);
		\draw [style=gray] (3) to (4);
		\draw [style=gray] (3) to (5);
		\draw [style=gray] (0) to (5);
		\draw [style=gray] (0) to (4);
		\draw [style=gray] (4) to (1);
		\draw [style=gray] (6) to (1);
		\draw [style=gray] (6) to (2);
		\draw [style=gray] (5) to (2);
	\end{pgfonlayer}
\end{tikzpicture}
    }}

\begin{document}

\maketitle
\begin{abstract}
    We show that every planar triangulation on $n$ vertices has a maximal independent set of size at most $n/3$. This affirms a conjecture by Botler, Fernandes and Guti\'errez [Electron.\ J.\ Comb., 2024], which in turn would follow if an open question of Goddard and Henning [Appl.\ Math.\ Comput., 2020] which asks if every planar triangulation has three disjoint maximal independent sets were answered in the affirmative. Since a maximal independent set is a special type of dominating set (independent dominating set),  this is a structural strengthening of a major result by Matheson and Tarjan [Eur.\ J.\ Comb., 1996] that every triangulated disc has a dominating set of size at most $n/3$, but restricted to triangulations.

    \paragraph{Keywords:}
        Independent dominating sets, 
        Independent domination number, 
        Maximal independent set, 
        Planar triangulations. 
    \\ \noindent
    \textbf{2020 AMS Subject Classification} 
        05C69.
\end{abstract}

\section{Introduction}

All graphs considered in this paper are finite and undirected. A \emph{triangulated disc} is a planar graph embedded in a plane such that all its internal faces are bounded by a cycle on three edges and the external face is bounded by a simple cycle. A \emph{triangulation} is a triangulated disc in which the outer face is also bounded by a cycle on three edges. Unless mentioned otherwise, the triangulations and triangulated discs we consider are simple. Let $G(V,E)$ be a graph. For $S \subset V$, the \emph{open neighborhood} $N(S)$ of $S$ is the set of all vertices in $G$ which have at least one neighbor in $S$. The set $S$ is called \emph{dominating} in $G$ if $N(S) \cup S = V$. The size of a smallest dominating set is called the \emph{domination number} $\gamma(G)$ of $G$. The set $S$ is \emph{independent}, if no two vertices in $S$ are adjacent. The size of a largest independent set is called the \emph{independence number} $\alpha(G)$ of $G$. A set $S\subseteq V$ is an \emph{independent dominating set} if $S$ is both independent and dominating. The size of a smallest independent dominating set is called the \emph{independent domination number} $\iota(G)$) of $G$. While a largest independent dominating set is same as a largest independent set, a smallest independent dominating set has a different behavior from that of a smallest dominating set. For example, $\gamma(K_{n,n})=2$ but $\iota(K_{n,n})=n$. In this work, we are interested in a smallest independent dominating set, or in other words, a smallest maximal independent set.  

Matheson and Tarjan~\cite{MathTar} in 1996 proved that there exist three disjoint dominating sets in any triangulated disc. Hence, for every triangulated disc $G$ on $n$ vertices, $\gamma(G)\leq n/3$. They also proved that this upper bound is tight for triangulated discs, while they conjectured that this could be improved to $n/4$ for large enough triangulations. This remains a major open problem. \v{S}pacapan \cite{vspacapan2020domination} in 2020 showed that every triangulation on $n > 6$ vertices has a dominating set of size at most $17n/53$. In 2024, Christiansen, Rotenberg and Rutschmann \cite{christiansen2024triangulations} showed that every triangulation on $n > 10$ vertices has a dominating set of size at most $2n/7$. We prove the following structural strengthening of Matheson and Tarjan's result for triangulations. 

\begin{theorem}\label{MainTheorem}
Every $n$-vertex triangulation has an independent dominating set of size at most $n/3$.  
\end{theorem}

Goddard and Henning~\cite{goddard2020independent} asked whether there exist three disjoint independent dominating sets in every triangulation. This will imply that, for any triangulation $G$, $\iota(G) \leq n/3$. Botler, Fernandes and Guti\'errez~\cite{botler2023independent} recently proved that, for every triangulation $G$ on $n$ vertices, $\iota(G) < 3n/8$ and if the minimum degree is at least five, then $\iota(G) \leq n/3$. Based on these results, they conjectured $\iota(G) \leq n/3$ for every triangulation. Goddard and Henning~\cite{goddard2020independent} also construct an infinite family of triangulations where the size of any independent dominating set is at least $6n/19$. Small triangulations like the triangle and the octahedron have $\iota(G) = n/3$. However, it is not clear whether the upper bound of $n/3$ can be improved for large enough $n$.

The study of independent domination (as well as domination) in graphs traces its origin to chessboard puzzles. In 1862, Carl Ferdinand von Jaenisch, a chess player and theorist, in his  famous three-volume treatise on chess strategies~\cite{jaenisch1862traite}, wanted to find out the minimum number of mutually non-attacking queens that can be arranged on an $8 \times 8$ chessboard so that every square on the board is attacked by at least one of these queens. A century later, in 1962, these questions were formalised in to the language of modern graph theory by  Berge~\cite{berge1962theory} and Ore~\cite{ore1962theory}. The term independent domination and the notation $\iota(G)$ were introduced by Cockayne and Hedetniemi in 1974~\cite{cockayne1974independence,cockayne1977towards}. Interestingly, the answer to Jaenisch's question is $5$, which is the same as the size of a smallest dominating set of queens on an $8 \times 8$ chessboard. MacGillivray and Seyffarth in~\cite{macgillivray2004bounds} proved that if $G$ is a connected graph on $n$ vertices with chromatic number $k \geq 3$, then $\iota(G) \leq ( k - 1 ) n/k - ( k - 2 )$. Combining this with the Four Color Theorem~\cite{AppHak,AppHakKoc}, we get $\iota(G) \leq 3n/4 - 2$ for any planar graph $G$. In the same paper, they also show that the upper bound can be improved to $\lceil n/3 \rceil$ if we restrict to planar graphs with diameter $2$. See the survey article by Goddard and Henning~\cite{goddard2013independent} for more results on independent domination of graphs.

Domination number is a widely studied parameter on many graph classes. The monograph by Haynes et al.~\cite{haynes2013fundamentals} gives a comprehensive reference on domination.  Reed~\cite{reed1996paths}, Sohn and Yuan~\cite{sohn2009domination}, and Xing et al.~\cite{xing2006domination} established upper bounds of $3n/8$, $4n/11$ and $5n/14$ respectively on the domination number for $n$-vertex graphs with minimum degree at least three, four and five. The domination number of planar graphs have received special attention. MacGillivray and Seyffarth~\cite{macgillivray1996domination} established an upper bound of three and ten respectively on the domination number of planar graphs with diameter two and three. Goddard and Henning~\cite{goddard2002domination} improved upon this and showed that there is only one planar graph of diameter two with domination number three; they also showed that every sufficiently large planar graph of diameter three has domination number at most seven. 

\emph{Independence number} $\alpha(G)$ is the size of a maximum independent set in $G$. The survey paper by Dainyak and Sapozhenko~\cite{DainyakSapozhenko} provides a comprehensive catalog of the results on independent sets. Caro~\cite{caro1979new} and Wei~\cite{wei1981lower} independently proved that for any graph $G$, $\alpha(G)\geq \sum_{v\in V(G)} (1+d_G(v))^{-1}$. There were some notable studies and improvements to this result in specific graph classes like triangle free graphs, $r$-colorable graphs, triangle free planar graphs etc. By the Four Color Theorem, we know that for any planar graph $G$, $\alpha(G) \geq n/4$. This result is tight. In 1973, while the Four Color Theorem was still unproven, Erd\H{o}s~\cite{bergeGraphesetHypergraphes} conjectured that $\alpha(G) \geq n/4$ for any $n$-vertex planar graph $G$. This is an invitation to prove this without using the Four Color Theorem. Albertson~\cite{ALBERTSON197684} in 1976, without using the Four Color Theorem, proved that every $n$-vertex planar graph $G$ has $\alpha(G) \geq 2n/9$. Carnston and Rabern~\cite{cranston2016planar} improved this to $3n/13$ in 2016. 
 
\subsection{Terminology and notation}
Let $G$ be a graph. The vertex-set and the edge-set of $G$ are denoted respectively by $V(G)$ and $E(G)$. The open neighborhood (resp. closed neighborhood) of a vertex $v$ in graph $G$ is denoted by $N_G(v)$ (resp. $N_G[v]$). The \emph{degree} $d_G(v)$ of a vertex $v$ in $G$ is $|N_G(v)|$. 
A graph is \emph{planar} if it can be embedded on the plane in such a way that no two edges cross each other. A plane graph $G$ is a planar graph together with such an embedding.  We denote the set of vertices lying on the boundary of a face $f$ as $V(f)$. A vertex of degree exactly $k$, at least $k$ and at most $k$ in $G$ are respectively termed \emph{$k$-vertex}, \emph{$k^+$-vertex} and \emph{$k^-$-vertex}. A cycle of length exactly $k$, at least $k$ and at most $k$ in $G$ are respectively termed \emph{$k$-cycle}, \emph{$k^+$-cycle} and \emph{$k^-$-cycle}. A face with exactly $k$, at least $k$ and at most $k$  edges in its boundary are termed respectively \emph{$k$-face}, \emph{$k^+$-face} and \emph{$k^-$-face}. We use the notations $P_k$, $K_k$ and $K_{k,k}$ to denote a path on $k$ vertices, a complete graph on $k$ vertices and a complete bipartite graph with $k$ vertices in each part respectively.

\section{Proof of Theorem~\ref{MainTheorem} }

 We will first show that a smallest (in terms of number of vertices) counterexample to Theorem~\ref{MainTheorem} cannot contain facial cycles formed by one $4$-vertex and two vertices of degree $4$ or $5$ (Lemma~\ref{facial_cycle}). We will then exploit this structure and the Four Color Theorem to obtain a partial proper four coloring of this counterexample with the key property that every $5^-$-vertex is guaranteed to see all four colors in its closed neighborhood (Lemma~\ref{coloring}). We will then show that one color class from this partial coloring can be expanded into an independent dominating set of size less than $n/3$, which is a contradiction. Bounding the size of this extension requires a technical lemma (Lemma~\ref{vertex_cover}) and an easy consequence of Euler's formula (Observation~\ref{independent_set}). We will prove and set aside Observation~\ref{independent_set} and then take up the lemmas in order.

\begin{observation} \label{independent_set}
   Every independent set of $6^+$-vertices in an $n$-vertex triangulation has size at most $(n-2)/3$.
\end{observation}
\begin{proof}
    Let $G$ be a triangulation on $n$ vertices and let $I$ be an independent set of $6^+$-vertices in $G$. Let $H$ be a plane graph obtained by deleting $I$ from $G$. Hence $|V(H)| = n-|I|$. Since all vertices in $I$ have degree at least six in $G$, $H$ contains only $3$-faces and $6^+$-faces. Let $f_3$ and $f_{6^+}$ denote the number of $3$-faces and $6^+$-faces respectively. By Euler's formula we have.
    
    \begin{equation} \label{eq:1}
            |V(H)| - 2 = |E(H)| - (f_3 + f_{6^+})
    \end{equation}
    
    By the standard double counting, $2|E(H)| \geq 3f_3 + 6f_{6^+}$. Also from the construction of $H$, we have $|V(H)| = n - |I|$ and $f_{6^+} = |I|$. Substituting these in equation~(\ref{eq:1}) 
    we get
     \begin{equation}
        |I| \leq \frac{ n-2 - f_3/2}{3}\leq \frac{n-2}{3}.
    \end{equation}     
\end{proof}


Now we begin the proof of Theorem~\ref{MainTheorem}. We will fix an arbitrary smallest counterexample (in terms of number of vertices) as $G$ and analyze it in detail. The first task is to prove Lemma~\ref{facial_cycle}. 
\begin{lemma} \label{facial_cycle}
If $G$ is a smallest counterexample to Theorem~\ref{MainTheorem}, then $G$ does not contain a facial cycle $(v_1,v_2,v_3)$ such that $4 =  d(v_1) \leq d(v_2) \leq d(v_3) \leq 5$.
\end{lemma}

The proof strategy is as follows. Suppose there exists a facial cycle $t = (v_1,v_2,v_3)$ in $G$ with $4 =  d(v_1) \leq d(v_2) \leq d(v_3) \leq 5$. We  delete the set $T = \set{v_1, v_2, v_3}$ from $G$ to obtain a subgraph $G^-$ with a new face $f$. We then carefully triangulate $f$ to obtain a triangulation $G'$ on at most $|V(G)| - 3$ vertices. In one case, we have to delete one more vertex before doing this triangulation. Finally, we show that an independent dominating set $D'$ in $G'$ can be extended to an independent dominating set $D$ of $G$ by adding at most one more vertex. This gives a contradiction since we can choose $D'$ to be of size $(|V(G')|-2)/3 \leq (|V(G)|-2)/3 -1$ and hence $|D| \leq (|V(G)|-2)/3$. The details of how to triangulate $f$ depend on the structure of $f$ while the selection of the additional vertex depends on $f$ and $D' \cap V(f)$. These cases constitute the bulk of the proof. 

\begin{proof}  
Let $t$, $T$, $G^-$ and $f$ be as defined in the proof strategy above. Let $B$ denote the closed walk bounding $f$ and let $s = \sum_{i=1}^3 d(v_i)$. We will use $G'$ to refer to the smaller triangulation to be constructed, $D'$ to refer to a smallest independent dominating set of $G'$ and $v_{sp}$ to denote the special vertex (if any) to be identified which when added to $D'$ will give an independent dominating set $D$ of $G$.
    
Each $v_i \in T$ has exactly two neighbors from $T$ itself in $G$. This contributes a total of six to $s$ and hence $s - 6$ edges go from $T$ to $V(B)$. If we contract the $3$-face $t$ to a single vertex $v_t$ (removing the self loops but retaining multiple edges), we get $s-6$ faces incident to $v_t$ of which exactly three are $2$-faces (the faces which contained an edge of $t$) while the remaining are $3$-faces. Since each step in the walk $B$ contributes to one $3$-face incident to $v_t$, the total length of $B$ is $s - 9$. There are only three possibilities for $(d(v_1),d(v_2),d(v_3))$, viz., $(4,4,4)$, $(4,4,5)$, and $(4,5,5)$. 

\paragraph{Case 1 ($d(v_1) = d(v_2) = d(v_3)= 4$).}
Here $s = 12$ and length $l_B$ of the closed walk $B$ is $12 - 9 = 3$. The only possible closed walk of length three in a simple graph is a $3$-cycle. So $B$ is a $3$-cycle. Here $G^-$ itself is a triangulation which is smaller than $G$, so we choose $G'=G^-$. If a smallest independent dominating set $D'$ of $G'$ itself dominates $G$, then we are done. Otherwise, there exists at least one vertex $v \in T$ which is not dominated by $D'$. We choose $v$ as $v_{sp}$. 

\paragraph{Case 2 ($d(v_1) =d(v_2) =4$, $d(v_3)= 5$).} 
In this case $s = 13$ and $l_B = 4$. Since $B$ is a closed walk of length four bounding a face in a simple graph $G$, $B$ can either be a $4$-cycle or along a $P_3$. Note that a path cannot divide a plane into two regions. So in any planar graph, if a $P_3$ is the boundary of a face $f$, then $f$ is the only face in the entire graph and $P_3$ is its only boundary. So $G^-$ is a $P_3$. Since $T$ was the only set of vertices deleted from $G$, $G$ is a triangulation on six vertices. Hence $v_3$, being a $5$-vertex, will dominate all of $G$. 

\begin{figure}[t]
    \centering
    
    \begin{subfigure}[t]{0.45\textwidth}
        \centering
     \begin{adjustbox}{max width=.9\textwidth}
    \begin{tikzpicture}{a}
	\begin{pgfonlayer}{nodelayer}
		\node [style={nofill_node}] (0) at (-2, 0) {$v_1$};
		\node [style={nofill_node}] (1) at (2, 0) {$v_2$};
		\node [style={nofill_node}] (2) at (0, -2.5) {$v_3$};
		\node [style={nofill_node}] (3) at (0, 2.75) {$b_{12}$};
		\node [style={nofill_node}] (4) at (-4, 0) {$b_{31}$};
		\node [style={nofill_node}] (5) at (0, -4.5) {$b_3$};
		\node [style={nofill_node}] (6) at (4, 0) {$b_{23}$};
	\end{pgfonlayer}
	\begin{pgfonlayer}{edgelayer}
		\draw (0) to (2);
		\draw (2) to (1);
		\draw (0) to (1);
		\draw (2) to (5);
		\draw (2) to (4);
		\draw (2) to (6);
		\draw (4) to (5);
		\draw (5) to (6);
		\draw (4) to (3);
		\draw (3) to (6);
		\draw (0) to (4);
		\draw (0) to (3);
		\draw (3) to (1);
		\draw (1) to (6);
	\end{pgfonlayer}
\end{tikzpicture}
\end{adjustbox}
\caption{Case when $B$ is a $4$-cycle}
\label{C4}
\end{subfigure}
    \begin{subfigure}[t]{0.45\textwidth} 
        \centering
         \begin{adjustbox}{max width=.9\textwidth}
        \begin{tikzpicture}{b}
	\begin{pgfonlayer}{nodelayer}
		\node [style={nofill_node}] (0) at (0, 1) {$v_1$};
		\node [style={nofill_node}] (1) at (2, -1) {$v_2$};
		\node [style={nofill_node}] (2) at (-2, -1) {$v_3$};
		\node [style={nofill_node}] (3) at (1.75, 3) {$b_{12}$};
		\node [style={nofill_node}] (4) at (-2, 3) {$b_{31}$};
		\node [style={nofill_node}] (5) at (-4, -1) {$b_3$};
		\node [style={nofill_node}] (6) at (0, -4.25) {$b_{23}$};
		\node [style={nofill_node}] (7) at (4, -1) {$b_2$};
	\end{pgfonlayer}
	\begin{pgfonlayer}{edgelayer}
		\draw (0) to (2);
		\draw (2) to (1);
		\draw (0) to (1);
		\draw (2) to (5);
		\draw (2) to (4);
		\draw (2) to (6);
		\draw (4) to (5);
		\draw (5) to (6);
		\draw (4) to (3);
		\draw (0) to (4);
		\draw (0) to (3);
		\draw (3) to (1);
		\draw (1) to (6);
		\draw (3) to (7);
		\draw (7) to (6);
		\draw (1) to (7);
	\end{pgfonlayer}
\end{tikzpicture}

\end{adjustbox}
 \caption{Case when $B$ is a $5$-cycle}
 \label{C5}
    \end{subfigure}
   \caption{Two of the forbidden structures in a smallest counterexample to Theorem~\ref{MainTheorem}.}
    \label{fig:enter-label}
\end{figure}
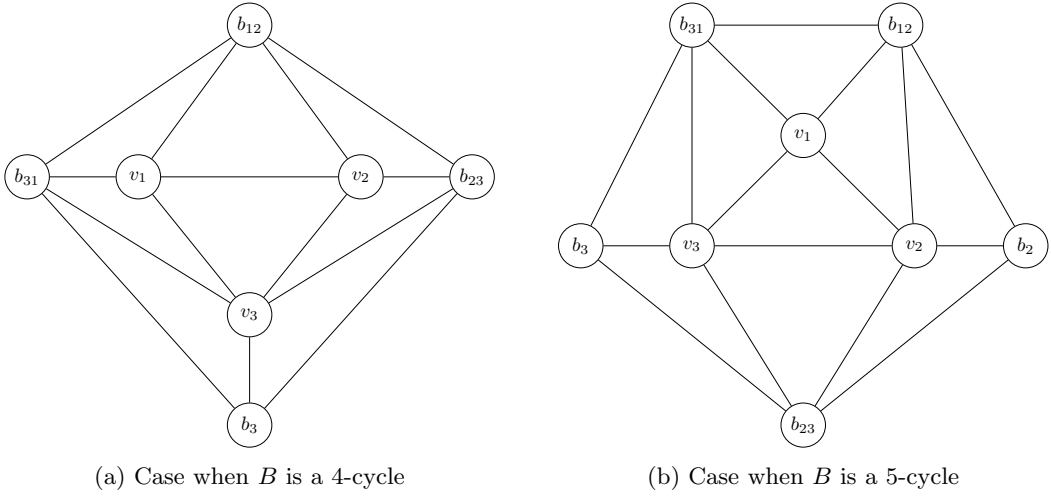

If $B$ is a $4$-cycle, then $B = (b_{12}, b_{23}, b_3,b_{31})$, where $b_3$ is adjacent to only $v_3$ in $T$ while $N(b_{ij}) \cap T = \{v_i, v_j\}$ (See Figure~\ref{C4}). Note that if a vertex $x$ in $B$ is adjacent to all three vertices in $T$, then it forms a $K_4$ in which one of the vertices of $T$ will have degree 3, a contradiction.  
In $f$, at least one pair of alternate vertices is not adjacent in $G^-$. This is because we cannot have both these (diagonal) edges through the exterior of $f$  without crossing. We obtain a triangulation $G'$ by adding this diagonal to $G^-$. If $D'$, a smallest dominating set in $G'$, dominates $G$ then we are done. If $D' \cap V(B) = \emptyset$, we select any one vertex $v_i \in T$ as $v_{sp}$. For $\{i,j,k\}=\{1,2,3\}$, if $D' \cap V(B)= \{b_{ij}\}$ then $\{b_{ij},v_k\}$ is independent and dominates all the vertices of $B$ and $T$ and hence we can choose $v_{sp} = v_k$. If $D' \cap V(B)=\{b_3\}$ then $\{b_{3},v_1\}$ is independent and dominates all the vertices of $B$ and $T$ and hence we can choose $v_{sp} = v_1$.
If $|D'\cap V(B)|=2$, then $D'\cap V(B)$ is either $\{b_{12},b_3\}$ or $\{b_{31},b_{23}\}$. In both these cases, $v_1$, $v_2$ and $v_3$ are dominated. Notice a subtle point that $D'$ is also a dominating set of $G^-$, since neither endpoint of the extra edge added for triangulating $f$ appears in $D'$. Hence, $D'$ is also a dominating set of $G$. 


\paragraph{Case 3 ($d(v_1) =4$, $d(v_2) = d(v_3)= 5$).} In this case, $s = 14$ and $l_B = 5$. Here too, there are only two possible configurations for the boundary $B$. Since $B$ is a closed walk of length five bounding a face in a simple graph, it can be either a $5$-cycle or along a $3$-cycle with a pendent vertex attached to one of its vertices. When $B$ is a $5$-cycle, the five vertices on $B$ appear cyclically in the order $(b_{12},b_2,b_{23},b_3,b_{31})$, where the subscripts indicate the neighbors of the vertex in $T$ as in the previous case (See Figure~\ref{C5}). 
We triangulate $G^- = G\setminus T$ to get $G'$ by adding two chords in the face $f$ of $G^-$. We can always do this without creating a multi-edge. Since $G$ does not contain a $K_5$, there is at least one chord missing in $f$. Once we draw this chord, we are left with a $4$-face. This can be triangulated by adding a diagonal as in the previous case to get $G'$ without creating any multi-edges. If a smallest independent dominating set $D'$ in $G'$ dominates $G$, we are done. If $D' \cap V(B) = \emptyset$, then $D'$ is also a dominating set in $G^-$. We select any vertex $v\in T$ as $v_{sp}$. Next we consider the cases when $\card{D' \cap V(B)} = 1$ (summarized in Table~\ref{tab:1}). For $\{i,j,k\}=\{1,2,3\}$, if $D' \cap V(B)= \{b_{ij}\}$ then $\{b_{ij},v_k\}$ is independent and dominates all the vertices of $B$ and $T$ and hence we can choose $v_{sp} = v_k$. For $\{j,k\}=\{2,3\}$, if $D' \cap V(B)=\{b_{j}\}$ then $\{b_{j},v_k\}$ is independent and dominates all the vertices of $B$ and $T$ and hence we can choose $v_{sp} = v_k$. 
If $\card{D' \cap V(B)} = 2$, then any two non-adjacent vertices from a $5$-cycle will dominate the entire $5$-cycle. Hence, $D'$ is a dominating set in $G^-$ also. The only case when two non-adjacent vertices of $B$ do not together dominate $T$ is when they are $b_2$ and $b_3$. In this case, $v_1$ is the only vertex which is not dominated and we select $v_1$ as $v_{sp}$. 

    \begin{table}[ht]
        \centering
        \begin{tabular}{|c|c|c|}
            \hline
            $D'\cap V(B)$ & $U$ & $v_{sp}$ \\\hline 
            $\set{b_{12}}$ & $\set{v_3,b_3,b_{23}}$ & $v_3$\\\hline 
            $\set{b_{2}}$ & $\set{v_1,v_3,b_3,b_{31}}$ & $v_3$\\\hline 
            $\set{b_{23}}$ & $\set{v_1,b_{12},b_{31}}$  & $v_1$\\\hline 
            $\set{b_{3}}$ & $\set{v_1,v_2,b_2,b_{12}}$ & $v_2$\\\hline 
            $\set{b_{31}}$ & $\set{v_2,b_2,b_{23}}$ & $v_2$\\\hline 
        \end{tabular}
        \caption{Choice of $v_{sp}$ when $B$ is a $5$-cycle and $\card{D' \cap V(B)} = 1$. $U$ is the largest set of vertices that may not be dominated by $D'$}
        \label{tab:1}
    \end{table}

The remaining part of the third case is when $B$ is along a $3$-cycle with a pendent vertex attached to one of its vertices. Let the closed walk be $(b_{12}, b_{2}, b_{23}, b_3, b_{31})$. We denote the pendent vertex as $x$ and its parent as $y$. Note that $y$ is visited twice in the closed walk. Hence, two vertices of the closed walk map to $y$. Since $x$ has at least three neighbors and its only neighbor outside $T$ is $y$, $x$  cannot be of type $b_i$. 
We obtain the triangulation $G'$ by deleting all the vertices $v_i \in T$ and the pendent vertex $x$ from $G$. The new $3$-face created will be denoted as $f'$. If a smallest independent dominating set $D'$ in $G'$ dominates $G$, we are done. Assume $D'$ does not have any vertices from the boundary $B'$ of $f'$. We know $x$ is of type $b_{ij}$. So we can select either $v_i$ or $v_j$ from $T$ as $v_{sp}$, which will dominate $x$ and $T$. Else, $D'$ can have only one vertex from $B'$. In this case, we select $v_{sp}$ based on the vertex on the closed walk which gets mapped to $x$ as described in Table~\ref{tab:2}.  If $b_{23}$ is mapped to $x$, then $b_2$ and $b_3$ are mapped to $y$. So the other two vertices on $B'$ are $b_{12}$ and $b_{31}$. Here, if $y$ is in $D'$, then it already dominates $x$, $v_2$ and $v_3$. So we choose $v_1$ which is the only undominated vertex as $v_{sp}$. If $b_{12}$ (resp. $b_{31}$) is in $D'$, then we choose $v_3$ (resp. $v_2$) as $v_{sp}$ which dominates $v_{sp}$ and $x$. If $b_{12}$ is mapped to $x$, then  $b_{31}$ and $b_{2}$ should be mapped to $y$. If $b_{31}$ is mapped to $x$, then $b_{3}$ and $b_{12}$ should be mapped to $y$. In both cases, if $y$ is in $D'$, then $D'$ dominates $G$. If any other vertex from $B'$ is in $D'$, then it should be $b_2$ or $b_3$ or $b_{23}$, hence we choose $v_1$ as $v_{sp}$.
\end{proof} 

    \begin{table}[ht]
    \setlength{\tabcolsep}{5mm} 
    \def\arraystretch{2.8} 
        \centering
        \begin{tabular}{|c|c|c|c|c|}
        \hline
             $x$ & $B$ & $D'\cap V(B)$ & $U$ & $v_{sp}$\\
             \hline
             \multirow{3}{*}{$b_{23}$} & \multirow{3}{*}{ \begin{adjustbox}{max width=.3\textwidth}{\begin{tikzpicture}
	\begin{pgfonlayer}{nodelayer}
		\node [style={nofill_node}] (0) at (0, 3) {$y$};
		\node [style={nofill_node}] (1) at (-4, -2) {$b_{12}$};
		\node [style={nofill_node}] (2) at (4, -2) {$b_{31}$};
		\node [style={nofill_node}] (3) at (0, 1.5) {$x$};
		\node [style={nofill_node}] (4) at (-1, 0) {\large$v_2$};
		\node [style={nofill_node}] (5) at (1, 0) {\large$v_3$};
		\node [style={nofill_node}] (6) at (0, -1) {\large$v_1$};
		\node  (7) at (3, 3) {\large$y = b_2 = b_3$};
		\node  (8) at (3, 2.25) {\large$x=b_{23}$};
	\end{pgfonlayer}
	\begin{pgfonlayer}{edgelayer}
		\draw (1) to (0);
		\draw (0) to (2);
		\draw (1) to (2);
		\draw (0) to (3);
		\draw [style=gray] (4) to (5);
		\draw [style=gray] (5) to (6);
		\draw [style=gray] (6) to (4);
		\draw [style=gray] (3) to (4);
		\draw [style=gray] (3) to (5);
		\draw [style=gray] (0) to (5);
		\draw [style=gray] (0) to (4);
		\draw [style=gray] (4) to (1);
		\draw [style=gray] (6) to (1);
		\draw [style=gray] (6) to (2);
		\draw [style=gray] (5) to (2);
	\end{pgfonlayer}
\end{tikzpicture}}\end{adjustbox}}  & $\set{y}$ & $\set{v_1}$ & $v_1$ \\\cline{3-5}
                                       &                                & $\set{b_{12}}$&$\set{v_3,x}$  &$v_3$ \\\cline{3-5}
                                       &                                & $\set{b_{31}}$&$\set{v_2,x}$  &$v_2$ \\\hline

           \multirow{3}{*}{$b_{12}$} & \multirow{3}{*}{ \begin{adjustbox}{max width=.3\textwidth}{\begin{tikzpicture}
	\begin{pgfonlayer}{nodelayer}
		\node [style={nofill_node}] (0) at (0, 3) {\large$y$};
		\node [style={nofill_node}] (1) at (-4, -2) {\large$b_3$};
		\node [style={nofill_node}] (2) at (4, -2) {\large$b_{23}$};
		\node [style={nofill_node}] (3) at (0, 1.5) {\large $x$};
		\node [style={nofill_node}] (4) at (0, 0.25) {\large$v_1$};
		\node [style={nofill_node}] (5) at (1, -1) {\large$v_2$};
		\node [style={nofill_node}] (6) at (-1, -1) {\large$v_3$};
		\node (7) at (3, 3) {\large{$y = b_2 = b_{31}$}};
		\node  (8) at (3, 2.25) {\large{$x=b_{12}$}};
	\end{pgfonlayer}
	\begin{pgfonlayer}{edgelayer}
		\draw (1) to (0);
		\draw (0) to (2);
		\draw (1) to (2);
		\draw (0) to (3);
		\draw [style=gray] (4) to (5);
		\draw [style=gray] (5) to (6);
		\draw [style=gray] (6) to (4);
		\draw [style=gray] (3) to (4);
		\draw [style=gray] (3) to (5);
		\draw [style=gray, bend left=15, looseness=1.25] (0) to (5);
		\draw [style=gray, bend right] (0) to (4);
		\draw [style=gray] (6) to (1);
		\draw [style=gray] (6) to (2);
		\draw [style=gray] (5) to (2);
		\draw [style=gray, bend left=15, looseness=1.25] (6) to (0);
	\end{pgfonlayer}
\end{tikzpicture}}\end{adjustbox}}  & $\set{y}$ & $\emptyset$ & $-$ \\\cline{3-5}
                                       &  & $\set{b_{3}}$&$\set{x,v_1,v_2}$  &$v_1$ \\\cline{3-5}
                                       &  & $\set{b_{23}}$&$\set{x,v_1}$  &$v_1$ \\\hline

           \multirow{3}{*}{$b_{31}$} & \multirow{3}{*}{ \begin{adjustbox}{max width=.3\textwidth}{\begin{tikzpicture}
	\begin{pgfonlayer}{nodelayer}
		\node [style={nofill_node}] (0) at (0, 3) {\large$y$};
		\node [style={nofill_node}] (1) at (-4, -2) {\large$b_{23}$};
		\node [style={nofill_node}] (2) at (4, -2) {\large$b_{2}$};
		\node [style={nofill_node}] (3) at (0, 1.5) {\large $x$};
		\node [style={nofill_node}] (4) at (0, 0.25) {\large$v_1$};
		\node [style={nofill_node}] (5) at (1, -1) {\large$v_2$};
		\node [style={nofill_node}] (6) at (-1, -1) {\large$v_3$};
		\node  (7) at (3, 3) {\large{$y = b_{12} = b_{3}$}};
		\node  (8) at (3, 2.25) {\large{$x=b_{31}$}};
	\end{pgfonlayer}
	\begin{pgfonlayer}{edgelayer}
		\draw (1) to (0);
		\draw (0) to (2);
		\draw (1) to (2);
		\draw (0) to (3);
		\draw [style=gray] (4) to (5);
		\draw [style=gray] (5) to (6);
		\draw [style=gray] (6) to (4);
		\draw [style=gray] (3) to (4);
		\draw [style=gray, bend left=15, looseness=1.25] (0) to (5);
		\draw [style=gray, bend left] (0) to (4);
		\draw [style=gray] (6) to (1);
		\draw [style=gray] (5) to (2);
		\draw [style=gray, bend left=15, looseness=1.25] (6) to (0);
		\draw [style=gray] (3) to (6);
		\draw [style=gray] (5) to (1);
	\end{pgfonlayer}
\end{tikzpicture}}\end{adjustbox}}  & $\set{y}$ & $\emptyset$ & $-$ \\\cline{3-5}
                                       &  & $\set{b_{2}}$&$\set{x,v_1,v_3}$  &$v_1$ \\\cline{3-5}
                                       &  & $\set{b_{23}}$&$\set{x,v_1}$  &$v_1$ \\\hline

        \end{tabular}
        \caption{Choice of $v_{sp}$ when $B$ is a $3$-cycle with a pendant vertex and $\card{D' \cap V(B)} = 1$. $U$ is the set of vertices not dominated by $D'$}
        \label{tab:2}
    \end{table}
    With the help of Lemma~\ref{facial_cycle} we will now define a partial proper $4$-coloring $\psi$ of the counterexample $G$ in which each $5^-$-vertex $v$ sees all four colors in its closed neighborhood (which includes $v$ and all its neighbors). For a coloring $\psi$ of $G$, we denote the set of colors used in the closed neighborhood of a vertex $v$ as $\psi[v]$.
    
\begin{lemma} \label{coloring}
    If $G$ is a smallest counterexample to Theorem~\ref{MainTheorem}, there exists a partial proper $4$-coloring $\psi$ of $G$, such that
    \begin{enumerate}[(i)]      
        \item any uncolored vertex has degree exactly four, \label{coloring-1}
        \item for every vertex $v$, $|\psi[v]| \geq 3$, \label{coloring-2}
        \item for every $5^-$-vertex $v$, $|\psi[v]| = 4$. \label{coloring-3}
 \end{enumerate}
\end{lemma}

\begin{proof}
    Since $K_3$ satisfies Theorem~\ref{MainTheorem}, $G$ is a triangulation on at least $4$ vertices and hence has minimum degree at least $3$. Consider all $4$-vertices in $G$ which are not adjacent to any $3$-vertex. Let $I$ be a maximal independent set of such vertices. We construct a new triangulation $G'$ (possibly non-simple) from $G$ to start our coloring. If $I$ is empty, $G$ itself is considered as $G'$. Else, delete $I$ from $G$ to get $G^-$. Since $I$ is an independent set of $4$-vertices, deletion of each vertex in $I$ creates a $4$-face in $G^-$. Let $f$ be such a $4$-face in $G^-$ and $B$ be its boundary. We make two observations about $B$. By the choice of $I$, no $3$-vertex of $G$ is on $B$. Further, by Lemma~\ref{facial_cycle}, two $5^-$-vertices of $G$ do not appear as adjacent vertices on $B$. Hence there exist two diagonally opposite $6^+$-vertices on $B$. Triangulate $G^-$ by adding an edge joining two $6^+$-vertices in every $4$-face of $G^-$ to obtain a triangulation $G'$. We allow the possibility that $G'$ may have multi-edges. A $4$-cycle in $G'$ corresponding to the boundary of a $4$-face in $G^-$ is called \emph{special}. 
    
    We start with a proper $4$-coloring of $G'$ and then extend it to a partial $4$-coloring of $G$. In a proper $4$-coloring of $G'$, each special $4$-cycle gets either three or four colors. Among all the proper $4$-colorings of $G'$, let $\phi$ be a coloring where the maximum number of special $4$-cycles are $3$-colored. Consider this same coloring $\phi$ of $V(G) \setminus I$ in $G$. Let $u$ be a vertex in $I$. If there are only three colors in $N_G(u)$, then give $u$ the fourth color.
    Else, leave it uncolored. We call this partial proper $4$-coloring of $G$ as $\psi$. Since the uncolored vertices (if any) are in $I$ and $I$ consists of only $4$-vertices,  $\psi$ satisfies~(\ref{coloring-1}). 

    Since $G$ is a triangulation, any vertex $v \in V(G)$ is part of a triangle $uvw$. If the vertices $u$, $v$ and $w$ are all colored then $|\psi[v]| \geq 3$. Since $I$ is an independent set, there can be only at most one uncolored vertex in any triangle of $G$. If $v$ itself is uncolored, then it was left so only because it saw all four colors in its neighborhood and hence $|\psi[v]| = 4$. The only remaining case is when $u$ or $w$, say $u$, is uncolored. Recall that every uncolored vertex in $\psi$ is a $4$-vertex with a $4$-colored $4$-cycle around it. So $v$ is part of a $4$-colored $4$-cycle around $u$ and thus sees three colors from this $4$-cycle. Hence, $\psi$ satisfies~(\ref{coloring-2}).
     
    We call a vertex $v$ \emph{happy} if $|\psi[v]|$ = 4. We will argue that every $5^-$-vertex of $G$ is happy. By choice of $I$, it is guaranteed that no vertex in the closed neighborhood of a $3$-vertex is in $I$ and hence every $3$-vertex is part of a $4$-colored $K_4$. Since all the four vertices of a $4$-colored $K_4$ are happy, all $3$-vertices and their neighbors are happy. The open neighborhood of any $4$-vertex $v$ in $I$ has either four or three colors. In the latter case, $v$ is colored with the fourth color and so $v$ is also happy. By the choice of $I$, any remaining $4$-vertex is adjacent to a vertex in $I$. Next, we consider both $4$-vertices and $5$-vertices adjacent to $I$.
    
    Let $v$ be any $5^-$-vertex adjacent to a vertex $u$ in $I$ and let $B_u = (v, v_2, v_3, v_4)$ be the $4$-cycle around $u$. The triangulation $G'$ contains the edge $v_2v_4$ (since the endpoints of new edges are $6^+$-vertices) and hence $\phi(v_2) \neq \phi(v_4)$. Thus $v$ sees three colors from $B_u$. If $u$ is colored, then it should be the fourth color, making $v$ happy. If $u$ is uncolored, then we will show that $v$ has a neighbor with the fourth color in $G$. If $v$ is part of any $3$-colored special $4$-cycle, then the above argument shows that $v$ is happy. Suppose $v$ was not a part of any special $4$-cycle and it did not have a neighbor with the fourth color in $G'$, then recoloring $v$ with the fourth color will give a proper coloring of $G'$ with more $3$-colored special $4$-cycles, which is a contradiction to the choice of $\phi$. 
    
    The only $5^-$-vertices to be analyzed are the $5$-vertices that are not adjacent to any vertex in $I$. If $v$ is such a vertex, then it has a fully colored odd cycle around it. Every odd cycle in a proper coloring gets at least three colors. So $N_G[v]$ gets all four colors making $v$ happy. This concludes the proof that $\psi$ satisfies~(\ref{coloring-3}).
\end{proof}

 Let $\psi$ be a partial proper $4$-coloring of $G$ which satisfies the conclusions of Lemma~\ref{coloring}. The coloring $\psi$ partitions $V(G)$ into five subsets $C_1, C_2, C_3, C_4, \overline{C}$. Each $C_i$ is the set of vertices with color $i$ and $\overline{C}$ is the set of uncolored vertices. Let $U_i$, $i \in [4]$, be the set of vertices of $G$ not in $C_i$ and with no neighbor in $C_i$. Let $U=\cup_{i=1}^4U_i$. We make two claims about $U$. From Lemma~\ref{coloring}~(\ref{coloring-1}) and~(\ref{coloring-3}), it can be observed that $U\cap \overline{C} = \emptyset$. Similarly, from Lemma~\ref{coloring}~(\ref{coloring-2}), we know when $i\neq j$, $U_i \cap U_j = \emptyset$. We call the set of edges between $U_i$ and $U_j$, $i \neq j$ as \emph{bad edges} $E_B$. Let $V_B$ be the set of all end points of bad edges. We call the graph $G_B = (V_B,E_B)$ the \emph{bad subgraph} of $G$. We refer to the $4$-cycle around an uncolored vertex as a \emph{critical cycle}.

 \begin{observation} \label{Obs:1}
    Every bad edge $e$ of $G$ is a part of two critical cycles.
\end{observation}
\begin{proof}
    Every edge in a triangulation is a part of two faces, hence $e$ is also part of two faces, say $uxv$ and $uyv$. If either $x$ or $y$ is colored then $u$ and $v$ see three common colors in their closed neighborhood which contradicts the fact that $e$ is a bad edge. So $e$ is a part of two critical cycles. 
\end{proof}

 \begin{lemma} \label{vertex_cover}
     Let $G$ be a smallest counterexample to Theorem~\ref{MainTheorem} and $\psi$ be a partial proper $4$-coloring satisfying the conclusions of Lemma~\ref{coloring} with $\overline{C}$ being the set of uncolored vertices. The bad subgraph $G_B$ of $G$ has a vertex cover of size at most $\card{\overline{C}}$. 
 \end{lemma}

\begin{proof}

 Let $H_1,\ldots,H_k$ be the connected components of $G_B$. For each $H_l$, $l \in [k]$, we prove the following claim. We say that a color $i$ is \emph{missing} at a vertex $v$ if no vertex in $N[v]$ is colored $i$, i.e., $v \in U_i$.

\begin{claim} \label{claim:bipartite}
    $H_l$ is bipartite with each part being the union of two color classes of $\psi$ restricted to $H_l$.
\end{claim}

Let $x$ be a vertex in $H_l$. Let us call $\psi(x)$ as $a$, the color missing at $x$ as $c$ and the two remaining colors as $b$ and $d$. Let $V_0 = (C_a \cap U_c) \cup (C_c \cap U_a)$  and  $V_1 = (C_b \cap U_d) \cup (C_d \cap U_b)$. Let $u$ be any vertex in $V_0$ and $v$ be any neighbor of $u$ in $H_l$. Since one of the colors in $\set{a,c}$ is $\psi(u)$ and the other is missing at $u$, $v$ has to be colored $b$ or $d$. The vertices $u$ and $v$ cannot be missing the same color (by the definition of a bad edge) and since $v$ cannot be missing $\psi(u)$, $v$ will be missing $b$ or $d$. Hence $v$ is in $V_1$. Similarly, we can show that for any vertex $u$ in $V_1$ all its neighbors in $H_l$ belong to $V_0$.  Since $H_l$ is connected, every vertex in $H_l$ belongs to either $V_0$ or $V_1$ and every edge in $H_l$ is between $V_0$ and $V_1$. Hence, $H_l$ is bipartite with each part being a union of two color classes. Note that both these parts are vertex covers of $H_l$.
  
Let $J$ be a vertex cover of $G_B$ which is constructed by adding one part of each $H_l$.  

\begin{claim}
    $\card{J} \leq \card{\overline{C}}$.   
\end{claim}

Consider the bipartite subgraph $H$ of $G$ defined as follows. The two parts of $H$ are $J$ and $\overline{C}$. We retain an edge $xy \in E(G)$ from $\overline{C}$ to $J$ only if $y$ is the endpoint of a bad edge in the critical cycle around $x$. We ignore all other edges. Consider any vertex $v$ in $J$. Since every bad edge is part of two critical cycles (Observation~\ref{Obs:1}), $v$ is adjacent to at least two vertices in $\overline{C}$. Now consider any vertex $u$ in $\overline{C}$. The vertex $u$ is only connected to end points of bad edges in the critical cycle around it. Only one end point of any bad edge is there in $J$, since it is constructed by adding one part of each $H_l$. So $u$ has at most two neighbors in $H$. Since every vertex in $J$ has at least two neighbors in $\overline{C}$ and every vertex in $\overline{C}$ has at most two in $H$, a simple double counting shows that $\card{J} \leq \card{\overline{C}}$.
\end{proof}

We proceed to prove Theorem~\ref{MainTheorem}. Let $J$ be a vertex cover of $G_B$ of size at most $\card{\overline{C}}$.  
For each $i\in[4]$, consider a maximal independent set $I_i$ of $U_i\setminus J$. The set $I_i$ dominates $U_i \setminus J$. Expand $I_i$ to a maximal independent set $I_i^+$ of $U_i$ by greedily adding vertices from $U_i \cap J$. The set $I_i^+$ will dominate the entire $U_i$. Hence, the set $D_i = C_i\cup I_i^+$ is an independent dominating set of $G$. Since $G[U\setminus J]$ does not have any bad edges, there are no edges between $I_i$ (subset of $U_i$) and $I_j$ (subset of $U_j$) when $i\neq j$. Hence,  $I=\cup_{i=1}^4I_i$ is an independent set. By Lemma~\ref{coloring} (\ref{coloring-3}), we know that every vertex in $U$ and hence $I$ has degree at least six. From Observation~\ref{independent_set}, the size of $I$ is at most $({n-2})/{3}$. So we have

    \begin{equation}
        \begin{split}
            \sum_{i=1}^4 |D_i| & = \sum_{i=1}^4 |C_i| + \sum_{i=1}^4 |I_i^+| \\
                        & \leq  \sum_{i=1}^4 |C_i| + \sum_{i=1}^4 |I_i| + \sum_{i=1}^4 |J \cap U_i|\\
                        & = n-|\overline{C}| + |I| + |J|\\
                        & \leq n-|\overline{C}| + \frac{n-2}{3} + |\overline{C}|\\
                        & = \frac{4n-2}{3}
        \end{split}
    \end{equation}

Hence, the smallest $D_i$ has size less than $n/3$. This contradiction completes the proof of Theorem~\ref{MainTheorem}.

\section{Concluding Remarks}

The triangle and the octahedron are triangulations which satisfy $\iota(G) = |V(G)|/3$. However, we are unable to find an infinite family of triangulations with this property. Let $f(n)$ denote the maximum independent domination number among all $n$-vertex triangulations. Goddard and Henning~\cite{goddard2020independent} construct, for positive integer $k$, a triangulation $G$ on $n = 19k - 12$ vertices with $\iota(G) \geq 6k - 4$, a lower bound which approaches $6n/19$ for large $k$\footnote{However, they reported it as $5n/19$ and this was improved to $2n/7$ (which is weaker than $6n/19$) by Botler et al.~\cite{botler2023independent}.}. Hence $\frac{6}{19} - o(1) \leq \frac{f(n)}{n} \leq \frac{1}{3}$. Closing this gap is an immediate task.

We can construct an infinite family of non-simple triangulations where $\iota(G) \geq |V(G)|/3$.  Start with a $2k$-vertex triangulation $T$, which contains a perfect matching $M$. Replace every edge $xy$ of the matching $M$ with a gadgets made of a $2$-cycle $(x,y,x)$ and a $2$-length path $(x,z,y)$ inside it to obtain a new triangulation $G$ on $n = 3k$ vertices. Any dominating set $D$ of $G$ should contain at least one vertex from each of the gadgets. Hence $|D| \geq k$. So we have $\iota(G) \geq \gamma(G) \geq n/3$. However, we do not have a matching upper bound for non-simple triangulations. This is another interesting problem.

  
The question by Goddard and Henning, whether there exist three disjoint independent dominating sets in every triangulation, remains unsettled and is quite intriguing. We also wonder whether the upper bound $\iota(G) \leq n/3$ can be extended to triangulated discs. If so, it will be a structural strengthening of the quantitatively tight result of Matheson and Tarjan that $\gamma(G) \leq n/3$ for any triangulated disc $G$. 

\bibliographystyle{ams}
\bibliography{bibtex}
\end{document}